\documentclass[12pt]{amsart}
\usepackage{fullpage,url,amssymb,enumerate,colonequals}

\usepackage{mathrsfs} % for \mathscr (script letters)

\usepackage{MnSymbol}
\usepackage{extarrows}
\usepackage{lscape}
\usepackage[all,cmtip]{xy}
\usepackage[OT2,T1]{fontenc}

\usepackage{color}

\usepackage[
        colorlinks, citecolor=darkgreen,
        backref,
        pdfauthor={Nuno Freitas}, % add other authors
]{hyperref}

\usepackage{array, booktabs, comment, multirow, todonotes}

\newcommand{\F}{\mathbb{F}}

\newcommand{\Q}{\mathbb{Q}}
\newcommand{\Z}{\mathbb{Z}}

\newcommand{\T}{\mathbb{T}}

\newcommand{\rhobar}{{\overline{\rho}}}
\newcommand{\calO}{\mathcal{O}}

\newcommand{\fp}{\mathfrak{p}}

\newcommand{\fq}{\mathfrak{q}}
\newcommand{\fN}{\mathfrak{N}}
\newcommand{\fL}{\mathfrak{L}}

\DeclareMathOperator{\Frob}{Frob}

\DeclareMathOperator{\Norm}{Norm}

\DeclareMathOperator{\Tr}{Tr}

\newcommand{\vv}{\upsilon}
\newcommand{\GL}{\operatorname{GL}}

\numberwithin{equation}{section}

\newtheorem{theorem}{Theorem}[section]

\newtheorem{corollary}[theorem]{Corollary}
\newtheorem{proposition}[theorem]{Proposition}

\theoremstyle{definition}

\theoremstyle{remark}
\newtheorem{remark}[theorem]{Remark}

\definecolor{darkgreen}{rgb}{0,0.5,0}

\begin{document}

\title{Revisiting the Fermat-type equation $x^{13} + y^{13} = 3z^7$}

\author{Nicolas Billerey}
\address{Universit\'e Clermont Auvergne, CNRS, LMBP, F-63000 Clermont-Ferrand, France.
}
\email{nicolas.billerey@uca.fr}

\author{Imin Chen}

\address{Department of Mathematics, Simon Fraser University\\
Burnaby, BC V5A 1S6, Canada } \email{ichen@sfu.ca}

\author{Lassina Demb\'el\'e}
\address{Department of Mathematics, King's College London, Strand, London WC2R 2LS, UK}
\email{lassina.dembele@kcl.ac.uk}

\author{Luis Dieulefait}

\address{Departament de Matem\`atiques i Inform\`atica,
Universitat de Barcelona,
G.V. de les Corts Catalanes 585,
08007 Barcelona, Spain}
\email{ldieulefait@ub.edu}

\author{Nuno Freitas}
\address{
Instituto de Ciencias Matem\'aticas, CSIC, 
Calle Nicol\'as Cabrera
13--15, 28049 Madrid, Spain}
\email{nuno.freitas@icmat.es}

\date{\today}


\thanks{Billerey was supported by the ANR-23-CE40-0006-01 Gaec project. Chen was supported by NSERC Discovery Grant RGPIN-2017-03892. Dieulefait and Freitas were partly supported by the PID 2022-136944NB-I00 grant of the MICINN (Spain).}

%\tableofcontents

\begin{abstract} 
We solve the Fermat-type equation
\[ x^{13} + y^{13} = 3 z^7, \qquad \gcd(x,y,z) = 1 \]
combining a unit sieve, the multi-Frey modular method, level raising, computations of systems of eigenvalues modulo 7 over a totally real field, and results for reducibility of certain Galois representations.
\end{abstract}

\maketitle

\section{Introduction}

In \cite{BCDDF} we gave three different and independent proofs of Theorem~\ref{thm:main} below. Our purpose with that paper was to not only complete the missing case $p=7$ in \cite[Theorem~2]{BCDF} but also to add new tools to the modular method that allow to deal with small concrete exponents that often survive the standard elimination techniques.

\begin{theorem}\label{thm:main}
There are no non-trivial primitive integer solutions to the equation
\begin{equation}\label{fte:13-13-7}
x^{13} + y^{13} = 3 z^7.
\end{equation}
\end{theorem}
Recall that a solution~\((x, y, z) = (a, b, c)\) of equation~\eqref{fte:13-13-7} is \emph{non-trivial} if it satisfies~\(abc \neq 0\) and we call it \emph{primitive} if~\(\gcd(a, b, c) = 1\). We recently found a bug in the code used for the proof of this theorem given in~\cite[\S 7]{BCDF} which implements a unit sieve. Correcting the bug revealed a fundamental obstruction to complete the proof using the unit sieve alone (see \S\ref{sec:descent-unit} for details), and the objective of this paper is to complete this proof. To achieve this, we will introduce new techniques for the elimination step of the modular method that apply when dealing with small exponents. In particular, our methods allow us to avoid prohibitively large spaces of Hilbert newforms (see Remark~\ref{rem:F}).

It is worth remarking the importance of having more tools  within  the modular method to deal with small exponents, since these tools tend to apply to various (small) exponents
and being less computationally expensive than more classical methods. For example,
it is natural to wonder whether one could use Chabauty methods to prove Theorem~\ref{thm:main} by determining the rational points on the hyperelliptic curves associated with \eqref{fte:13-13-7} from the constructions in \cite{BDF25,DahmenSiksek}. Note that, even in case of success, this would depend on GRH as is the case in these works.

Following the strategy initiated in~\cite{BCDF} and~\cite{BCDDF}, a crucial step to prove Theorem~\ref{thm:main} is to extend~\cite[Theorem 7]{BCDF} to the exponent $p=7$.
Once this is achieved, Theorem~\ref{thm:main} follows immediately from~\cite[Theorem~4.1]{BCDDF}.
Thus we must prove
the following.

\begin{theorem}\label{thm:discard-g}
Let~$(a,b,c)$ be a non-trivial primitive solution to equation~\eqref{fte:13-13-7}.

Then $13 \mid a+b$ and $4 \mid a+b$.
\label{T:thm7}
\end{theorem}
The proof of this theorem combines several ingredients. In Section~\ref{sec:descent-unit}, the conclusion $4 \mid a+b$ follows from a combination of a unit sieve with the multi-Frey modular method; this sieve also yields a unit obstruction to the conclusion $13 \mid a+b$.
To deal with this obstruction, we first introduce in Section~\ref{sec:levelraising} a method to force auxiliary level lowering/raising primes in the Frey curve; then in Sections~\ref{S:computation}~and~\ref{sec:reducible} we combine computations of systems of eigenvalues mod~$7$ and inductively show that the residual representations of the Frey curve must level lower to reducible representations, yielding a contradiction.

\paragraph*{\bf Electronic resources} For the computations needed in this paper, we used  {\tt Magma} \cite{magma}. Our programs are posted in \cite{programs}, where there is included descriptions, output transcripts, timings, and machines used.

%%%%%%%%%%%%%%%%%%%%%
\paragraph*{\bf Acknowledgments}   Freitas is grateful to Max Planck Institute for Mathematics in Bonn for its hospitality and financial support.

\section{A modular unit sieve}
\label{sec:descent-unit}
In this section, we  attempt to prove Theorem~\ref{thm:discard-g}
by combining the study of units in classical descent with
local restrictions on the solutions coming from the multi-Frey approach in~\cite{BCDF}. However, we shall shortly see that we are only able to derive the conclusion $4 \mid a+b$ whilst
the 13-adic conclusion of Theorem~\ref{thm:discard-g} fails due to a unit obstruction\footnote{This unit obstruction went unnoticed in~\cite{BCDDF} due to a bug in the code; in particular, since there were no apparent obstructions due to the bug, the proof incorrectly finished.}.

Suppose $(a,b,c)$ is a non-trivial
primitive solution to~\eqref{fte:13-13-7}.

Let $\zeta$ be a primitive 13th root of unity.  We then have the factorization in $\Z[\zeta]$,
\begin{align*}
  a^{13} + b^{13} & = (a+b) \prod_{i=1}^{12} (a + \zeta^i b) = 3 c^7.
\end{align*}
The integers \(a + b\) and~\(\frac{a^{13} + b^{13}}{a + b}\) are coprime away from~\(13\) (see, e.g., \cite[Lemma~2.1]{DahmenSiksek}). Let~\(\ell\not\equiv 1\pmod{13}\) be a prime number dividing~\(a^{13} + b^{13}\). Since~\(a\) and~\(b\) are coprime, we have that~\(\ell\nmid ab\) and hence there exists an integer~\(b'\) such that~\(bb'\equiv -1 \pmod{\ell}\).  Therefore we have~\((ab')^{13}\equiv 1\pmod{\ell}\) and since~\(\ell\not\equiv 1\pmod{13}\), it follows that~\(\ell\) divides~\(a + b\). In particular, we have that~$\gcd (3,a + \zeta b) = 1$. Furthermore, by classical descent, we have that
\begin{equation}
\label{unit-equation}
  a + \zeta b = \begin{cases}
    \epsilon \beta^7  & \text{ if } 13 \nmid a+b \\
    \epsilon (1 - \zeta) \beta^7 & \text{ if } 13 \mid a+b,
   \end{cases}
\end{equation}
where $\epsilon$ is a unit of $\Z[\zeta]$ and $\beta \in \Z[\zeta]$.
We only need to consider $\epsilon$ up to 7th powers, which means that there are initially $16,807$ possible choices for $\epsilon$.

One can now reduce~\eqref{unit-equation} modulo a prime $\fq$ of $\Z[\zeta]$ above the rational 
prime~$q$. If $q$ is such that the order of the 
multiplicative group of the residue field at~$\fq$ is divisible 
by~$7$, the condition of $(a+\zeta b)\epsilon^{-1}$ or
$(a+\zeta b)\epsilon^{-1}(1-\zeta)^{-1}$
being a 7th power is non-trivial. Note that, in particular, the primes $q = 2, 11, 19, 23$ and $83$ satisfy this condition, a fact which will be used in the sequel.

We now consider the Frey curve $E = E_{a,b}$ over $\Q(\sqrt{13})$ defined in~\cite[\S 7.1]{BCDF} (see the file {\tt 13-curveE.m} in~\cite{programs} for an implementation in {\tt Magma}). From~\cite[Proposition~9]{BCDF} there is a mod 7 congruence
between $E$ and one of $E_{1,0}$, $E_{1,1}$, $E_{1,-1}$ or a certain Hilbert newform~$g$ over~\(\Q(\sqrt{13})\) of parallel weight~\(2\), trivial character, with field of coefficients $\Q_g = \Q(\sqrt{2})$. From the second and third paragraphs in the proof of \cite[Theorem~2]{BCDF} and \cite[Remark~7.4]{BCDF}
we conclude that
\begin{equation}\label{E:isos}
\rhobar_{E,7} \cong \rhobar_{E_{1,-1},7} \quad \text{ or } \quad
\rhobar_{E,7} \cong \rhobar_{g,\fp_7},
\end{equation}
where $\fp_7 \mid 7$ in $\Q_g$ is the prime in $\Q_g = \Q(\sqrt{2})$ such that $\sqrt{2} + 3 \in \fp_7$.

We can obtain $q$-adic information from the previous isomorphism; indeed, the isomorphism imposes constraints on the solutions $(a,b)$ modulo $q$ (in both cases of good or multiplicative reduction of $E_{a,b}$ at $q$) and hence on the unit $\epsilon$. We cannot obtain information from $q=2$ in the same way,
because $2$ divides the level of~$g$. Nevertheless, since
the multiplicative group of the residual field of $K$ at~$\fq=(2)$ has order a multiple of 7, this together with
$4 \nmid a + b$  also imposes restrictions on~$\epsilon$.

\begin{remark}
Observe that \cite[Proposition~6.1]{BCDDF} says that
$\rhobar_{E_{1,-1},7} \cong \rhobar_{g,\fp_7}$ hence
the information arising from the isomorphisms in~\eqref{E:isos} is the same. Our objective in this work is of course not to use this fact in the proof of Theorem~\ref{thm:discard-g}, as doing so is the content of one of the alternative proofs in {\it loc. cit.}.
%However we can at least use it to accelerate the computations required in this section; indeed, we avoid computing $g$ as we only need to work with $E_{1,-1}$ and accessing traces of Frobenius for the elliptic curve is faster.
\end{remark}
We now assume, for a contradiction, that $4 \nmid a+b$.
Note that $4 \nmid a+b$
is equivalent to $2 \nmid a + b$ given the shape of equation~\eqref{fte:13-13-7}. We sieve the set of possible units in the two cases of equation~\eqref{unit-equation},
assuming $2 \nmid a+b$. More precisely,
in both the cases $13 \nmid a+b$ and $13 \mid a+b$, using the primes
$q = 2, 11, 19, 23$ the set of units which passes all local conditions is empty (here the local conditions include the constraints on $(a,b)$ modulo~$q$ arising from the modular method as discussed above). See the file~{\tt ModularSieve.m} in~\cite{programs} for details. Thus~$4 \mid a+b$, as desired.

\begin{remark} \label{rem:F}
At this point we could continue the proof of Theorem~\ref{thm:discard-g} with a standard application of the modular method using the Frey curve $F_{a,b}$ introduced in~\cite[\S 7.2]{BCDF}. Indeed, assuming $4 \mid a+b$ and $13 \nmid a+b$,
from lemmas~9~and~11 in {\it loc. cit.}, we have $\rhobar_{F_{a,b},7} \simeq \rhobar_{f,\fp_7}$ where $f$
is a Hilbert newform  of parallel weight 2, trivial character and level $2 \cdot 3 \cdot \fq_{13}^2$ defined over the cubic real subfield~$K$ of the \(13\)th cyclotomic field, where $\fq_{13}$ is the unique prime in~$K$ above~13. The cuspidal and new subspaces are of dimensions 3823 and 2353, respectively. Computing the newforms in this space using {\tt Magma} takes about 90 hours, and yields 73 newforms including several with coefficient field of degree $d \in \{102, 108, 117, 153 \}$. In fields of such high degrees it is not computationally practical to factor ideals, so the refined elimination technique from~\cite[\S7.3]{BCDF} is not feasible. Therefore, we can only perform standard elimination by testing non-divisibility of norms by 7, but this fails to eliminate many newforms with high degree coefficient fields.
 \end{remark}

We can now assume $4 \mid a + b$ and $13 \nmid a + b$ aiming for another contradiction.
Using a similar sieve (see again the file~{\tt ModularSieve.m} in~\cite{programs}), now with the primes $q = 2, 11, 19, 23$ and~$83$ we conclude there is only one unit  surviving the sieve. More precisely, it is the unit~$\epsilon_0$ satisfying the relation
\begin{equation}\label{E:epsilon0}
\epsilon_0:=\frac{13^4}{(1-\zeta)^{48}} \qquad
13^4-\zeta 13^4= \epsilon_0 (1-\zeta)^{49}.
\end{equation}

We note that $(a,b) = (13^4,-13^4)$ does not satisfy $13 \nmid a+b$ but unfortunately when working modulo primes $q \neq 13$ this condition becomes invisible.  \footnote{This is precisely the error with the proof in \cite[\S 7]{BCDDF}. Due to a bug in the code performing the sieve, the unit $\epsilon_0$ was also eliminated and we did not notice the relation~\eqref{E:epsilon0}.}

\begin{remark}
This method of sieving units with modular information also appears in \cite{DahmenSiksek}, in the context of reducing the number of hyperelliptic curves to be considered for Chabauty. However, only the primes $q$ not dividing the level are considered there. Here we also sieve at $q=2$ which divides the level.
It should also be noted that the unit $\epsilon = 1$ does not pass the local conditions from the modular method (unlike the situation in \cite{DahmenSiksek}) because the equation~\eqref{fte:13-13-7} does not have the trivial solutions $\pm (1,0,1), \pm (0,1,1)$ due to the coefficient $3$.
\end{remark}

\section{Forcing level raising primes}
\label{sec:levelraising}

Let $(a,b,c)$ be a primitive solution to~\eqref{fte:13-13-7} satisfying $4 \mid a+b$ and $13 \nmid a+b$.

From the previous section, we have that
\begin{equation}\label{E:epsilon0eq}
 a+\zeta b = \epsilon_0 \beta^7 \quad \text{ with } \quad  \epsilon_0 =\frac{13^4}{(1-\zeta)^{48}}, \quad \beta \in \Z[\zeta]
\end{equation}
and we also know that sieving \eqref{E:epsilon0eq} modulo any prime $q$ will not discard $\epsilon_0$. Nevertheless, we still use the modular sieve above to gather further local information; indeed, having fixed an auxiliary prime~$q$, we run over all the pairs $(a,b)$ mod~$q$ and keep only the pairs compatible with~\eqref{E:isos} and~\eqref{E:epsilon0eq}. By doing this for several small primes, we find that for
$$
q \in \fL := \{5, 17, 19, 23, 29, 37, 41, 43, 61, 83, 89\}
$$
all the pairs $(a,b)$ mod~$q$ that survive also satisfy~$a+b \equiv 0 \pmod{q}$.

For the primes~\(q = 5, 17, 23, 29, 43, 61\), this conclusion already follows from~\eqref{E:isos}, but for $q =19, 37, 41, 83, 89$ we need~\eqref{E:epsilon0eq} as well\footnote{To speed up the computations for the primes~\(q = 29\) and~\(43\) (which are both~\(\equiv 1 \pmod{7}\)), we can use the fact that, under the assumption~\(13 \nmid a + b\), there exist (coprime) integers~\(c_1\) and~\(c_2\) such that~\(a + b = 3 c_1^7\) and~\(\frac{a^{13} + b^{13}}{a + b} = c_2^7\).}. See the file~{\tt LevelRaisingSieve.m} in~\cite{programs}.

We now associate to $(a,b,c)$ the Frey curve $F_{a,b} / K$ introduced in~\cite[\S 7.2]{BCDF},
where~$K$ is the cubic real subfield of the \(13\)th cyclotomic field.
From the previous discussion we can assume $q \mid a+b$ for all~$q \in \fL$.

From Theorem~8 and Lemmas~9~and~11 in {\it loc. cit.}, we have that~$\rhobar_{F_{a,b},7}$ is irreducible and
\begin{equation} \label{E:isos2}
\rhobar_{F_{a,b},7} \simeq \rhobar_{f,\fp_7}
\end{equation}
where $f$
is a Hilbert newform  of parallel weight 2, trivial character and level $2 \cdot 3 \cdot \fq_{13}^2$, where~$\fq_{13}$ is the unique prime in $K$ above~13 and $\fp_7 \mid 7$ is a prime in the field of coefficients of~$f$.
As explained in Remark~\ref{rem:F} it is not computationally feasible to obtain a contradiction to~\eqref{E:isos2} using standard elimination techniques.
Instead, it follows from Lemma~8 in {\it loc. cit.} and $q \mid a+b$ for $q \in \fL$,  that $F_{a,b}$ has multiplicative reduction at the primes~$\fq$ in $K$ above all the primes in $\fL$. This means that
level lowering modulo 7 is occurring to the representation
$\rho_{F_{a,b},7}$ at all these primes and hence, the newform $f$ in~\eqref{E:isos2} must simultaneously satisfy the level raising condition at the same primes. More precisely, we have
\begin{equation}\label{E:levelRaising}
 a_\fq(f) \equiv \pm (\Norm(\fq) + 1) \pmod{\fp_7} \quad \text{ for all primes } \fq \mid q \; \text{ with } q \in \fL.
\end{equation}
In a nutshell, we have used the unit sieve with modular information
to force level raising primes into the obstructing newforms. It is now very unlikely that a newform $f$ will satisfy~\eqref{E:levelRaising}, unless it has a reducible mod~$\fp_7$ representation, but reducible representations cannot occur in the isomorphism \eqref{E:isos2} because $\rhobar_{F_{a,b},7}$ is irreducible, giving a contradiction. 

In the remaining sections below, we will show that such an~$f$ indeed has reducible mod~$\fp_7$ representation.

\begin{remark}
  For the computational arguments in the following sections, we eventually only use a subset of the known level raising primes in $\fL$, namely the primes above $5$, and $83$.
\end{remark}

\section{Eliminating newforms modulo~7}
\label{S:computation}

Recall that~$K$ is the cubic real subfield of the \(13\)th cyclotomic field. Let $\fN = 2\cdot 3 \cdot \fq_{13}^2$.
We consider the set of Hecke operators $T_\fq$ in $S_2^{\rm new}(\fN ; \F_7)$
given by
\[
\mathcal{T}_1 = \left\{ \; T_\fq \; :  \Norm(\fq) \le 31 \text{ or } \Norm(\fq) = 83 \right\}                                                                                                                                                                                                                                                                   \]
and let $\mathbb{T}_1$ be the Hecke algebra generated by~$\mathcal{T}_1$.
Let $S_1$ be the socle of $S_2^{\rm new}(\fN, \F_7)$ viewed as a $\T_1$-module.
Note that among the primes in~$\mathcal{T}_1$ only the three primes above 5 and the three primes above~\(83\) belong to the set~$\fL$ of level raising primes. 

By~\eqref{E:isos2}, the obstructing form $f$ has a mod~$\fp_7$ representation isomorphic to that of an elliptic curve, hence its corresponding mod~$\fp_7$ system of eigenvalues lives in $\F_7$.

We computed $S_1$ and extracted the~\(43\) systems of eigenvalues in~$\F_7$. We remark that, when running {\tt Magma} to perform this computation, it is likely that some eigensystems will appear with multiplicity $> 1$ using the lattice provided by {\tt Magma}. Instead, applying an algorithm being developed by the authors, we are able to find a lattice (stored in the file {\tt Matrix.m}) such that the calculation yields multiplicity~$1$ for all $\F_7$-systems of eigenvalues. Indeed, using this lattice, the computation of \(\F_7\)-eigensystems takes about \(4\) hours using the file {\tt Certificate.m} given in~\cite{programs}\footnote{We note that due to random choices done by {\tt Magma} in order for the lattice computation being correct, we need to specify the initial seed provided in the file {\tt Certificate.m}.}. For convenience of the reader, we have saved the \(\F_7\)-eigensystems in the file {\tt Eigensystems.out}.

Running the file {\tt EliminationMod7.m}, we find that there are exactly four systems of eigenvalues $\mathcal{E}_1$, $\mathcal{E}_2$, $\mathcal{E}^\chi_1$ and $\mathcal{E}^\chi_2$
satisfying the congruence~\eqref{E:levelRaising}  simultaneously at the primes $\fq \mid 5 \cdot 83$; in fact, more is true:
\[
 a_\fq(\mathcal{E}_i) \equiv \Norm(\fq) + 1 \pmod{7} \quad \text{ and } \quad a_\fq(\mathcal{E}^\chi_i) \equiv \chi|_{G_K}(\fq) (\Norm(\fq) + 1) \pmod{7},
\]
for all the primes~$\fq$ such that $T_\fq \in \mathcal{T}_1$, where $\chi$ is the quadratic character of~$G_\Q$ fixing~$\Q(\sqrt{13})$. The eigensystems $\mathcal{E}_1, \mathcal{E}_2$ and $\mathcal{E}^\chi_1, \mathcal{E}^\chi_2$ can be distinguished by their eigenvalues at primes dividing~$\fN$, as given in Table~\ref{table:eigenvalues}.
This calculation serves as evidence that $\mathcal{E}_i$ and $\mathcal{E}^\chi_i$ arise from the reducible representations $1 \oplus \chi_7|_{G_K}$ and $(\chi \oplus \chi \chi_7 )|_{G_K}$, respectively, where $\chi_7$ is the mod~7 cyclotomic character; in fact, this is the content of the following theorem which we will prove in the next section.
\begin{theorem}\label{T:reducibleSystems}
For all primes $\fq$ in $K$ not dividing $2\cdot 3 \cdot \fq_{13}$, we have
\[
 a_\fq(\mathcal{E}_i) \equiv \Norm(\fq) + 1 \pmod{\fp_7} \quad \text{ and } \quad a_\fq(\mathcal{E}^\chi_i) \equiv \chi(\fq) (\Norm(\fq) + 1) \pmod{\fp_7}.
\]
\end{theorem}
Therefore $\mathcal{E}_i$ and~$\mathcal{E}^\chi_i$ do not arise from any newforms~$f$ satisfying~\eqref{E:isos2} as such a form has an irreducible mod~$\fp_7$ representation. This finishes the elimination of all newforms, completing the proof of Theorem~\ref{thm:discard-g} and hence Theorem~\ref{thm:main}.
\begin{table}[h]
    \centering
    \begin{tabular}{|c||c|c|c|c|}
     \hline
        $\fq \mid \fN$ & $\mathcal{E}_1$ & $\mathcal{E}_2$ & $\mathcal{E}^\chi_1$ & $\mathcal{E}^\chi_2$
        \\ \hline
        $2\calO_K$ & 1 & 1 & -1 & -1 \\ \hline
        $3\calO_K$ & 1 & -1 & 1 & -1 \\ \hline
        $\fq_{13}$  & 0 & 0 & 0 & 0 \\ \hline
    \end{tabular}
\caption{Eigenvalues at Steinberg primes}
    \label{table:eigenvalues}
\end{table}

\begin{remark}
The computation in the file {\tt Certificate.m} applies the socle method of \cite[Section~6]{BCDDF} which is a computationally efficient method used to prove two newforms are congruent modulo a concrete prime $\fp$. In this method, one needs to specify a Hecke invariant lattice in the space of newforms. When the corresponding eigensystem for a finite set of Hecke operators appears with multiplicity $1$, with respect to the chosen lattice, we succeed in proving the two newforms are congruent modulo $\fp$.
\end{remark}

\section{Newforms with reducible mod \texorpdfstring{$\fp_7$}{} representations}
\label{sec:reducible}

For the proof of Theorem~\ref{T:reducibleSystems} below we need various auxiliary results.

Let $\fN_0 = 2\cdot 3 \cdot \fq_{13}$ where again~\(\fq_{13}\) is the unique prime ideal above~\(13\) in the totally real cubic subfield~\(K\) of the \(13\)th cyclotomic field. In \cite[Proposition 3.2]{BCDDF} we proved that there are 5 pairs $(f,\fp_0)$ where $f$ is a Hilbert newform in $S_2(\fN_0)$ and $\fp_0$ is a prime in its coefficient field~$\Q_f$ with residue field $\F_{\fp_0} = \F_7$ such that $\rhobar_{f,\fp_0}$ is reducible under the assumption that $\rhobar_{f,\fp_0}$ is ramified at all primes $\fq \mid \fN_0$. (See the file~{\tt PubMatSection3.m} in~\cite{programs} for the computations supporting the assertions in Section~3 of~\cite{BCDDF}.) The following result gets rid of this ramification assumption.

\begin{proposition} \label{P:reducibleSmallLevel}
The representation $\rhobar_{f,\fp_0}$ is reducible for each of the five pairs above.
\end{proposition}
\begin{proof} For the computations of this proof, see the file~{\tt RedMod7.m} in~\cite{programs}. Let $(f,\fp_0)$ be any of the pairs in \cite[Proposition 3.2]{BCDDF}. For all primes $\fq \nmid \fN_0$ with $\Norm(\fq) < 100$
we check that $a_\fq(f) \equiv 1 + \Norm(\fq) \pmod{\fp_0}$.

For a contradiction, we assume that $\rhobar_{f,\fp_0}$ is irreducible. We can also assume that $\rhobar_{f,\fp_0}$ is unramified at least at one prime $\fq \mid \fN_0$ because otherwise the result is \cite[Proposition 3.2]{BCDDF}.

The possibilities for the Serre level $N(\rhobar_{f,\fp_0})$ are~$n\calO_K$ for $n = 1,2,3,6$ and $\fq_{13}$, $2\fq_{13}$, $3\fq_{13}$. From level lowering for Hilbert modular forms, there is a Hilbert newform $g \in S_2(N(\rhobar_{f,\fp_0}))$
such that $\rhobar_{f,\fp_0} \simeq \rhobar_{g,\fp_7}$ for some prime $\fp_7 \mid 7$ in $\Q_g$.

A quick calculation with {\tt Magma} shows that for
$N(\rhobar_{f,\fp_0}) \in \{ 1\calO_K, \fq_{13} \}$ there are no newforms. At level~\(2\calO_K\), there is a single Galois orbit of newforms labelled {\tt 3.3.169.1-8.1-a} in the LMFDB \cite{lmfdb} and it corresponds to the isogeny class of elliptic curves over~\(K\) with conductor~\(2 \calO_K\). This isogeny class in turn contains an elliptic curve labelled {\tt 8.1-a2} whose Mordell--Weil group is isomorphic to~\(\Z / 7 \Z\). In particular, the mod~\(7\) representation~\(\rhobar_{g,\fp_7}\) is reducible, hence contradicting our assumption.

We check that there is exactly one pair $(g_0,\fp_7')$ where $g_0$ is a newform of level $N(\rhobar_{f,\fp_0}) = 3\calO_K$ and $\fp_7' \mid 7$ in~$\Q_{g_0}$ such that, for all primes $\fq \nmid \fN$ of $\Norm(\fq) < 30$, we have~$a_\fq(g_0) \equiv 1 + \Norm(\fq) \pmod{\fp_7'}$. The newform~\(g_0\) is the form labelled {\tt 3.3.169.1-27.1-a} in LMFDB \cite{lmfdb} and it corresponds to the isogeny class {\tt 27.1-a} which again contains an elliptic curve with a \(K\)-rational \(7\)-torsion point. We conclude as before.

We conclude that $N(\rhobar_{f,\fp_0}) = 6\calO_K, 2\fq_{13}$ or~$3 \fq_{13}$. From the proof of Proposition~3.2 in \cite{BCDDF} there is an irreducible representation $\rhobar : G_\Q \to \GL_2(\overline{\F}_7)$ that extends $\rhobar_{f,\fp_0}$, that is $\rhobar|_{G_K} = \rhobar_{f,\fp_0}$ (this does not require the additional ramification hypothesis in {\it loc. cit.}).

We will now determine the Serre level $N = N(\rhobar)$, the nebentypus $\epsilon(\rhobar)$ and
Serre weight $k(\rhobar)$.

We have $k(\rhobar) = 2$ because $\rhobar|_{G_K} = \rhobar_{f,\fp_0}$ is of parallel weight~$2$ and 7 is unramified in~$K$.

Recall that $2$ and $3$ are inert in~$K$, and that $13$ is the only ramified prime in $K$.

Let~\(\fq\) be a prime ideal in~\(K\) above a rational prime~\(\ell\) dividing~\(N(\rhobar_{f,\fp_0})\).

If~\(\ell = 2, 3\), then~$\rhobar_{f,\fp_0}$ is a Steinberg representation at~$\fq$, thus
$\rhobar$ is also Steinberg at $\ell$, and hence~$\vv_\ell(N) = 1$. If~$\ell = 13$, the representation $\rhobar$ is either Steinberg or a twist of Steinberg by a character $\chi$ of~$G_\Q$
that becomes trivial over $K$. In the latter case, $\chi$ is of conductor~$13^1$ and the conductor of $\rhobar$ at~$13$ is $13^2$.
So twisting $\rhobar$ by $\chi$, if necessary, we can assume also that $\vv_{13}(N) = 1$ since both $\rhobar$ and
$\rhobar \otimes \chi$ restrict to $\rhobar_{f,\fp_0}$.
Thus $N$ is the squarefree product of the rational primes dividing the norm of~\(N(\rhobar_{f,\fp_0})\), that is, $N \in \{6,26, 39\}$.

Since $\rhobar$ is Steinberg at every prime~\(\ell \mid N\), it follows that
the nebentypus $\epsilon(\rhobar)$ is locally trivial at~$\ell$;
thus $\epsilon(\rhobar) = 1$ as there are no other ramified primes for $\rhobar$.

From Khare--Wintenberger (Serre's conjecture)~\cite{kw09}, there is a newform
$h \in S_2(N)^{\rm new}$
and a prime $\fp \mid 7$ in $\Q_h$ such that $\rhobar \simeq \rhobar_{h, \fp}$. For~\(N = 6\), the space~\(S_2(6)^{\rm new}\) is trivial and we get a contradiction.

Therefore, we have that~\(N = \ell \cdot 13\) with~\(\ell = 2, 3\). Note that for a rational prime $q$ that splits in $K$ and a prime $\fq \mid q$ in $K$ we have
\[
 \Tr(\rhobar_{f,\fp_0}(\Frob_\fq)) = \Tr(\rhobar(\Frob_\fq)) = \Tr(\rhobar_{h,\fp}(\Frob_\fq)) = \Tr(\rhobar_{h,\fp}(\Frob_q))  \pmod{\fp}.
\]
For each of $\ell=2,3$ we check this congruence at a few splitting primes $q$ and we conclude that $(h,\fp)$ is unique and satisfies $a_q(h) \equiv 1 + q \pmod{\fp}$. We now apply~\cite[Theorem~A]{martin} (which corresponds to the special case of classical modular forms in~Theorem~2.1 of~\emph{loc. cit.}) with~\(N_1 = \ell\) and~\(N_2 = 13\). We find that there is a cuspidal eigenform in~\(S_2(\ell \cdot 13)\) with reducible associated mod~\(7\) representation isomorphic to~\(1 \oplus \chi_7\) (up to semisimplification). Moreover, this form is necessarily new as~\(S_2(m)^{\rm new} = \{0\}\) for~\(m\) a proper divisor of~\(\ell \cdot 13\).
It follows that $h$ must be such form, hence $\rhobar_{h, \fp}$ is reducible, a contradiction with $\rhobar \simeq \rhobar_{h, \fp}$.

We have obtained a contradiction with each possible $N(\rhobar_{f,\fp_0})$ thus~$\rhobar_{f,\fp_0}$ is reducible.
\end{proof}

\begin{proposition}
\label{P:reducibleSmallLevel2}
Let $f \in S_2(\fN_0)^{\rm new}$ be a Hilbert newform such that $\rhobar_{f,\fp_7}$ is reducible for some prime $\fp_7 \mid 7$ in its coefficient field.
Then $\rhobar_{f,\fp_7}^{\rm ss} \simeq 1 \oplus \chi_7|_{G_K}$.
\end{proposition}
\begin{proof}

For $f$ as in the statement, let~$\F_{\fp_7}$ denotes the residual field of $\Q_f$ at $\fp_7$. By the reducibility assumption, there are characters $\theta, \theta' : G_K \to \F_{\fp_7}^*$ such that the semisimplification of~$\rhobar_{f,\fp_7}$ satisfies
$ \rhobar_{f,\fp_7}^{ss} \simeq \theta \oplus \theta'$ with $\theta \theta' = \chi_7|_{G_K}$.
For all primes $\fq \nmid 7$, we have $\theta'|_{I_\fq} = \theta^{-1}|_{I_\fq}$ (since $\chi_7$ is unramified away from $7$). In particular, the conductor exponent of~$\theta$ and~$\theta'$ is the same at each such prime, and these characters are unramified at primes where the conductor of~\(\rhobar_{f,\fp_7}\) has valuation 0 or odd. It follows that $\theta$ and $\theta'$ are unramified away from 7.

Note that $7$ is inert in $K$.
Since $\fN_0$ is coprime to~7 and $f$ is of weight 2, the $\fp_7$-adic representation~$\rho_{f,\fp_7}$ is Barsotti--Tate and arises in an abelian variety since $K$ has odd degree (see~\cite{zha01}). Therefore, by~\cite[Corollaire 3.4.4]{Raynaud}, the restriction to an inertia subgroup~$I_7 \subset G_K$
of $\rhobar_{f,\fp_7}$ is isomorphic to either $\chi_7|_{I_7} \oplus 1$ or $\psi \oplus \psi^p$, where~$\psi$ is a fundamental character of level~$2$.
In the latter case, since the tame character $\psi$ has image (of inertia) of order $7^2-1$, by local class field theory, we must have $48 \mid \# (\calO_K/7\calO_K)^\times = 7^3-1$, a contradiction.

Thus exactly one of~\(\theta,\theta'\) is unramified at~7 (while the other restricts to~\(I_7\) as the cyclotomic character); after swapping $\theta$ and $\theta'$ if needed, we can assume that $\theta$ is unramified at all finite primes. As $K$ has narrow class number 1, we have $\theta = 1$ and $\theta' = \chi_7|_{G_K}$, as desired.
\end{proof}

\begin{corollary}\label{C:reducibleSmallLevel}
 Let $(f, \fp_0)$ be one of the five pairs above. Then $\rhobar_{f,\fp_0}^{\rm ss} \simeq 1 \oplus \chi_7|_{G_K}$ and
 $$ (a_2(f),a_3(f),a_{\fq_{13}}(f)) \in \{ (1,-1,1),(1,1,1),(1,1,-1),(1,-1,-1) \}.$$
Moreover, there is no other pair in $S_2(\fN_0)$ such  that $a_\fq(f) \equiv 1 + \Norm(\fq) \pmod{\fp_0}$ for all primes $\fq \nmid \fN\cdot 7$ of $\Norm(\fq) \le 100$.
\end{corollary}
\begin{proof}The first claim is a direct consequence of Propositions~\ref{P:reducibleSmallLevel} and~\ref{P:reducibleSmallLevel2}, and the rest follows from a {\tt Magma} calculation.
\end{proof}

\subsection{Proof of Theorem~\ref{T:reducibleSystems}}
Let $\chi$ be the quadratic character of~$\Q(\sqrt{13})$. Let $(f,\fp_0)$ be a pair as in Corollary~\ref{C:reducibleSmallLevel} and set $g = f \otimes \chi$. The form $g$ is new at level~$\fN =  2\cdot 3 \cdot \fq_{13}^2$.
It follows that $(g,\fp_0)$ has
system of eigenvalues matching $(\chi \oplus \chi \chi_7 )|_{G_K}$ away from the level and satisfying
$$ (a_2(g),a_3(g),a_{\fq_{13}}(g)) \in \{ (-1,-1,0),(-1,1,0) \}.$$
From Table~\ref{table:eigenvalues} we see this proves the theorem for the systems of eigenvalues $\mathcal{E}^\chi_i$.

By the Deligne--Serre lemma, there is a newform $F_i \in S^{\text{new}}_2(\fN)$ giving rise to the $\F_7$-system of eigenvalues~$\mathcal{E}_i$.
We have $v_{\fq_{13}}(\fN) = 2$ and the conductor exponent of~$\chi|_{G_K}$ is 1, therefore, the twisted
form $F_i \otimes \chi|_{G_K}$ must be new at level $6\calO_K$, $\fN_0$ or~$\fN$, and give rise to the
eigensystem~$\mathcal{E}_i \otimes \chi|_{G_K}$. We already know from previous proofs that there is no system of eigenvalues at levels $6\calO_K$ and~$\fN_0$ matching $\mathcal{E}_i \otimes \chi|_{G_K}$ away from the level, therefore $\mathcal{E}_i \otimes \chi|_{G_K}$ appears at level~$\fN$ and so must be equal to~$\mathcal{E}^\chi_i$. Since we already know the theorem holds for $\mathcal{E}^\chi_i$, twisting it by $\chi|_{G_K}$ yields the theorem for~$\mathcal{E}_i$.\qed

\begin{remark}
At the end of~\S\ref{S:computation}, we have provided a certificate that allows us to assume that all the eigensystems  appear with multiplicity one. Note that we computed the $\F_7$-eigensystems for the Hecke operators in $\mathcal{T}_1$, therefore, if (say) $\mathcal{E}_1^\chi$ appeared with multiplicity 2 we would not be able to decide whether the two copies agree for all Hecke operators or if some prime not in~$\mathcal{T}_1$ distinguishes them. In particular, in the previous proof, the mod~$\fp_0$ eigensystem of the form~$g$ would only show that one of the two copies of the $\mathcal{T}_1$-eigensystem $\mathcal{E}_1^\chi$ is reducible.
\end{remark}

%\bibliography{biblio}{}
%\bibliographystyle{plain}

\end{document}